\documentclass{amsart}

\usepackage{mathrsfs}
\usepackage{graphicx}

\makeatletter

\newcommand{\Rmnum}[1]{\expandafter\@slowromancap\romannumeral #1@}
\makeatother

\newtheorem{theorem}{Theorem}[section]
\newtheorem{lemma}[theorem]{Lemma}

\theoremstyle{definition}

\newtheorem{corollary}[theorem]{Corollary}
\theoremstyle{remark}
\newtheorem{remark}[theorem]{Remark}
\newtheorem{question}[theorem]{Question}
\numberwithin{equation}{section}

\AtEndDocument{\bigskip{\footnotesize%
  \textsc{Department of Mathematics, Ningbo University, Ningbo, Zhejiang, People's Republic of China} \par
  \textit{E-mail address:} \texttt{kanjiangbunnik@yahoo.com, jiangkan@nbu.edu.cn} \par
}}
\AtEndDocument{\bigskip{\footnotesize%
  \textsc{Department of Mathematics, Ningbo University, Ningbo, Zhejiang, People's Republic of China} \par
  \textit{E-mail address:} \texttt{283186364@qq.com} \par
}}
\AtEndDocument{\bigskip{\footnotesize%
  \textsc{Department of Mathematics, Ningbo University, Ningbo, Zhejiang, People's Republic of China} \par
  \textit{E-mail address:} \texttt{740464942@qq.com} \par
}}
\AtEndDocument{\bigskip{\footnotesize%
  \textsc{Department of Mathematics, Ningbo University, Ningbo, Zhejiang, People's Republic of China} \par
  \textit{E-mail address:} \texttt{2017084162@qq.com} \par
}}
\begin{document}
\title{Multiple representations of real numbers on self-similar sets with overlaps}
\author{Kan Jiang, Xiaomin Ren, Jiali Zhu and Li Tian}

\begin{abstract}
Let $K$  be the attractor of the following IFS
 $$\{f_1(x)=\lambda  x, f_2(x)=\lambda x +c-\lambda,f_3(x)=\lambda x +1-\lambda\}, $$
 where  $f_1(I)\cap f_2(I)\neq \emptyset, (f_1(I)\cup f_2(I))\cap f_3(I)=\emptyset,$
and  $I=[0,1]$ is the convex hull of $K$. 
The main results of this paper are  as follows: 
$$\sqrt{K}+\sqrt{K}=[0,2]$$ if and only if 
$$\sqrt{c}+1\geq 2\sqrt{1-\lambda},$$
where $\sqrt{K}+\sqrt{K}=\{\sqrt{x}+\sqrt{y}:x,y\in K\}$.
If $c\geq (1-\lambda)^2$, then $$\dfrac{K}{K}=\left\{\dfrac{x}{y}:x,y\in K, y\neq 0\right\}=\left[0,\infty\right).$$ As a consequence, we  prove that the  following conditions are equivalent:
\begin{itemize}
\item [(1)] For any $u\in [0,1]$, there are some $x,y\in K$ such that $u=x\cdot y;$
\item [(2)] For any $u\in [0,1]$, there are some $x_1,x_2,x_3,x_4,x_5,x_6,x_7,x_8,x_9,x_{10}\in K$ such that $$u=x_1+x_2=x_3-x_4=x_5\cdot x_6=x_7\div x_8=\sqrt{x_9}+\sqrt{x_{10}};$$
\item [(3)] $c\geq (1-\lambda)^2$.
\end{itemize}
\end{abstract}

\maketitle{}

\section{Introduction}
There are many methods which can represent real numbers. For instance, the $\beta$-expansions \cite{KO,Akiyama,BakerG,KM,DajaniDeVrie,EHJ,MK,GS}, the
continued fractions \cite{KarmaOomen,KCor},  the {L}\"uroth
expansions \cite{KarmaCor1}, and so on.  In this paper, we shall analyze a new representation, that is, the 
 arithmetic representation of real numbers  in terms of some  self-similar sets with overlaps.  Firstly, let us introduce some fundamental definitions and results. 
Let  $A,B\subset\mathbb{R}$ be two non-empty sets. Define $$A*B=\{x*y:x\in A, y\in B\},$$ where $*$ is  $+, -,\cdot$ or $\div$ (if  $*=\div$,  then we assume $y\neq 0$).  We call $u=x*y$ an arithmetic representation in terms of $A$ and $B$. 
Steinhaus \cite{HS} proved that 
\begin{equation*}
C-C=\{x-y:x,y\in C\}=[-1,1], 
\end{equation*}%
where $C$ is 
the middle-third Cantor set.  Recently, Athreya, Reznick and Tyson \cite{Tyson} considered the multiplication on $C$, and proved
that
\begin{equation*}
17/21\leq \mathcal{L}(C\cdot C)\leq 8/9,
\end{equation*}%
where $\mathcal{L}$ denotes the Lebesgue measure. In \cite{XiKan3}, Jiang and Xi proved that $C\cdot C$ contains countably many intervals.  Moreover, they also came up with a sufficient condition such that  the image of $C\times C$ under some continuous functions contains a non-empty interior. The readers can find more results  in \cite{Yuki,PS, Hochman2012, Yoccoz, Shmerkin,Kan2019,SumKan}.
Athreya, Reznick and Tyson \cite{Tyson} also investigated the division of $C$, namely the set $C\div C$, and proved that $C\div C$ is  precisely the countably  union of some  closed intervals. In \cite{XiKan2}, Jiang and Xi  considered the representations of  real numbers  in $ C-C=[-1,1]$, i.e. let $x\in [-1,1]$, define 
\begin{equation*}
S_{x}=\left\{ \mathbf{(}y_{1},y_{2}\mathbf{)}:y_{1}-y_{2}=x,\,\, (y_1,y_2)\in C\times C\right\} .
\end{equation*}
and 
\begin{equation*}
U_{r}=\{x:\sharp(S_{x})=r\},  r\in \mathbb{N}^{+}. 
\end{equation*}
They proved that $\dim_{H}(U_r)=\dfrac{\log 2}{\log 3}$ if  $r=2^k$ for some $k\in \mathbb{N}$. Moreover, $$0<\mathcal{H}^{s}(U_1)<\infty, \mathcal{H}^{s}(U_{ 2^k})=\infty,k\in \mathbb{N}^{+},$$ 
where $s=\dfrac{\log 2}{\log 3}$. 
$U_{3\cdot 2^k}$ is an infinitely countable set for any $k\geq1$, where $\dim_{H}$ and $\mathcal{H}^{s}$ denote the Hausdorff dimension and Hausdorff measure, respectively.  There are more general results in \cite{XiKan2}. 
In this paper, we shall analyze the following self-similar set with overlaps \cite{Hutchinson}, 
let   $K$ be   the attractor  generated by the following IFS,  $$\{f_1(x)=\lambda x, f_2(x)=\lambda x+c-\lambda,f_3(x)=\lambda x+1-\lambda\},$$ where  $f_1(I)\cap f_2(I)\neq \emptyset, (f_1(I)\cup f_2(I))\cap f_3(I)=\emptyset,$
and  $I=[0,1]$ is the convex hull of $K$.
The  self-similar set $K$  is a typical example which allows serious overlaps. Many people analyzed this example from various aspects. Furstenberg conjectured that  the following self-similar set
 with the equation 
 $$K_1=\dfrac{K_1}{3}\cup \dfrac{K_1+\alpha}{3} \cup \dfrac{K_1+2}{3}$$
 has Hausdorff dimension for any irrational number $\alpha.$ Hochman \cite{Hochman} made use of elegant methods from ergodic theory proving that this conjecture is correct. 
Keyon \cite{Keyon},
 Rao and Wen \cite{Rao} proved that $\mathcal{H}^{1}(K)>0$ if and only if $\lambda=p/q\in \mathbb{Q}$ with $p\equiv q\not\equiv (0\equiv 3)$. Feng  and Lau \cite{FengLau} gave a multifractal analysis of $K$ when $K$ satisfies the weak separation property.  Yao and Li \cite{YaoLi}  analyzed all the generating IFS's of $K$ when $c=2\lambda-\lambda^2$. 
 Ngai and Wang \cite{NW} gave a definition of   finite type condition, and  an algorithm that  can calculate the Haudorff dimension of $K$ when  $K$ is of finite type.  In \cite{DJKL, KarmaKan2}, Dajani et al. analyzed the points in $K$ with multiple codings, and gave some examples such that the set of points with exactly $3$ codings can be empty while the set of points with precisely  $2$ codings has the same Hausdorff dimension of the univoque set. In \cite{Xi}, Guo et al. in terms of some ideas from \cite{NW},  considered the bi-Lipschitz equivalence of overlapping self-similar sets. In \cite{JWX}, Jiang, Wang and Xi  gave a necessary condition such that  when $c=\lambda-\lambda^2$  the self-similar set $K$ is bi-Lipschitz equivalent to another self-similar set with  the strong separation condtion.
 In \cite{XiKan1}, Tian et al.  proved that $K\cdot K=[0,1]$ if and only if $c\geq (1-\lambda)^2$. This result is sharp.  Moreover, Jiang and Xi \cite{XiKan3} also proved  the following result. 
Suppose that  $f$ is  a continuous function
defined on an open set $U\subset \mathbb{R}^{2}$. Denote the image
\begin{equation*}
f_{U}(K,K)=\{f(x,y):(x,y)\in (K\times K)\cap U\}.
\end{equation*}%
If $\partial _{x}f$, $\partial _{y}f$ are continuous on $U,$ and there is a
point $(x_{0},y_{0})\in (K\times K)\cap U$ such that one of the following conditions is satisfied, 
\begin{equation*}
\max\left\{\dfrac{1-c-\lambda}{\lambda}, \dfrac{1-\lambda}{1-c}\right\}<\left\vert \frac{\partial
_{y}f|_{(x_{0},y_{0})}}{\partial _{x}f|_{(x_{0},y_{0})}}\right\vert <\dfrac{1}{1-c-\lambda},
\end{equation*}
or
\begin{equation*}
\max\left\{\dfrac{1-c-\lambda}{\lambda}, \dfrac{1-\lambda}{1-c}\right\}<\left\vert \frac{\partial
_{x}f|_{(x_{0},y_{0})}}{\partial _{y}f|_{(x_{0},y_{0})}}\right\vert <\dfrac{1}{1-c-\lambda},
\end{equation*}
then $f_{U}(K,K)$ has a non-empty interior.
We emphasize that it is difficult to obtain the above results if one utilizes the Newhouse thickness theorem \cite{SN}. The main reason is that   it is not easy to calculate the thickness of $K$ as there are very complicated overlaps in $K$. 
For the Assouad dimension of $K$  and the geodesic distance on $K\times K$, we refer to \cite{XiA,XiB,XiC,XiD,XiE,XiF,XiG,XiH,XiI,SBB,YaoYao,LTT11,XiC11}.  For the average weighted receiving time on the complex networks  or on $K$, the readers can find results  in  the papers \cite{Lin, D3,D4}.

We have mentioned many results concerining with $K$ from different perspective. In this paper we shall consider the  multiple representations, i.e.  addition, substract, multiplication,  division and square root,  on $K$. This is the main motivation of this paper. In fact, similar analysis appears in the setting of $\beta$-expansions. For instance, the   multiple $\beta$-expansions and simultaneous expansions are considered by Komornik, Pedicini, and Peth\H{o} \cite{KPP}, Dajani, Jiang and Kempton \cite{DJK}, Hare and Sidorov\cite{Hare2016, Hare2017}, Dajani et al. \cite{KarmaKan2,DJKL}. The  simultaneous expansions  are related with the interior of the associated self-affine sets.  For more applications of the  simultaneous expansions, see \cite{G}.  

 The following  are the main results of this paper. 
 \begin{theorem}\label{Main}
Let $K$  be the attractors of the following IFS
 $$\{f_1(x)=\lambda  x, f_2(x)=\lambda x +c-\lambda,f_3(x)=\lambda x +1-\lambda\}, $$
 where  $f_1(I)\cap f_2(I)\neq \emptyset, (f_1(I)\cup f_2(I))\cap f_3(I)=\emptyset,$
and  $I=[0,1]$ is the convex hull of $K$. If $c\geq (1-\lambda)^2$, then $$\dfrac{K}{K}=\left\{\dfrac{x}{y}:x,y\in K, y\neq 0\right\}=\left[0,\infty\right).$$
\end{theorem}
 \begin{theorem}\label{Main1}
Let $K$  be the attractors  defined in the above theorem. Then 
$$\sqrt{K}+\sqrt{K}=[0,2]$$ if and only if
$$\sqrt{c}+1\geq 2\sqrt{1-\lambda}.$$
\end{theorem}
Note that $\sqrt{c}+1\geq 2\sqrt{1-\lambda}$ implies $c\geq (1-\lambda)^2$, see Figure 1. Therefore, we have the following result. 
\begin{corollary}\label{Cor}
Let $K$  be the attractors  defined in the above theorem. Then the  following conditions are equivalent:
\begin{itemize}
\item [(1)] For any $u\in [0,1]$, there are some $x,y\in K$ such that $u=x\cdot y;$
\item [(2)] For any $u\in [0,1]$, there are some $x_1,x_2,x_3,x_4,x_5,x_6,x_7,x_8,x_9,x_{10}\in K$ such that $$u=x_1+x_2=x_3-x_4=x_5\cdot x_6=x_7\div x_8=\sqrt{x_9}+\sqrt{x_{10}};$$
\item [(3)] $c\geq (1-\lambda)^2$.
\end{itemize}
\end{corollary}
This paper is arranged as follows. In section 2, we give the proof of Theorem \ref{Main}. In section 3, we prove Theorem \ref{Main1}.  In section 4, we pose one problem. 
\section{Proofs of Theorem \ref{Main} and Corollary \ref{Cor}}
Let $H=[0,1]$. For any $(i_{1},\cdots ,i_{n})\in \{1,2,3\}^{n}$, we call $%
f_{i_{1},\cdots ,i_{n}}(H)=(f_{i_{1}}\circ \cdots \circ f_{i_{n}})(H)$ a
basic interval of rank $n,$ which has length $\lambda^n$.  Denote by $H_{n}$ the collection of all these basic intervals
of rank $n$. Let $J\in H_{n}$, then $\widetilde{J}=\cup _{i=1}^{3}I_{n+1,i}$%
, where $I_{n+1,i}\in H_{n+1}$ and $I_{n+1,i}\subset J$ for $i=1,2,3$. Let $%
[A,B]\subset \lbrack 0,1]$, where $A$ and $B$ are the left and right
endpoints of some basic intervals in $H_{k}$ for some $k\geq 1$,
respectively. $A$ and $B$ may not be in the same basic interval. Let $F_{k}$
be the collection of all the basic intervals in $[A,B]$ with length $\lambda^k,k\geq k_{0}$ for some $k_{0}\in \mathbb{N}^{+}$, i.e. the union of
all the elements of $F_{k}$ is denoted by $G_{k}=\cup _{i=1}^{t_{k}}I_{k,i}$%
, where $t_{k}\in \mathbb{N}^{+}$, $I_{k,i}\in H_{k}$ and $I_{k,i}\subset
\lbrack A,B]$. Clearly, by the definition of $G_{n}$, it follows that $%
G_{n+1}\subset G_{n}$ for any $n\geq k_{0}.$ Similarly, suppose that $M$ and
$N$ are the left and right endpoints of some basic intervals in $H_{k}$.
Denote by $G_{k}^{\prime }$ the union of all the basic intervals with length
$\lambda^k$ in the interval $[M,N]$, i.e. $G_{k}^{\prime }=\cup
_{j=1}^{t_{k}^{\prime }}J_{k,j}$, where $t_{k}^{\prime }\in \mathbb{N}^{+}$,
$J_{k,j}\in H_{k}$ and $J_{k,j}\subset \lbrack M,N]$.

Very useful is the following lemma. It   comes from \cite{Tyson} and \cite{XiKan1}. 
For the convenience of readers, we give the detailed proof. 
\begin{lemma}\label{lem}
Suppose $U\subset \mathbb{R}^2$ is a non-empty open set. 
Let $F:U\rightarrow \mathbb{R}$  be a continuous function$.$
Suppose $A$ and $B$ ($M$ and $N$) are the left and right endpoints of some
basic intervals in $H_{k_{0}}$ for some $k_{0}\geq 1$ respectively such that
$[A,B]\times \lbrack M,N]\subset U.$ Then $K\cap \lbrack A,B]=\cap _{n={k_{0}%
}}^{\infty }G_{n}$, and $K\cap \lbrack M,N]=\cap _{n={k_{0}}}^{\infty
}G_{n}^{\prime }$. Moreover, if for any $n\geq k_{0}$ and any two basic
intervals $I\subset G_{n}$, $J\subset G_{n}^{\prime }$ such that
\begin{equation*}
F(I,J)=F(\widetilde{I},\widetilde{J}),
\end{equation*}%
then $F(K\cap \lbrack A,B],K\cap \lbrack M,N])=F(G_{k_{0}},G_{k_{0}}^{\prime
}).$
\end{lemma}
\begin{proof}
By the construction of $G_{n}$ ($G_{n}^{\prime }$), i.e. $G_{n+1}\subset
G_{n}$ ($G_{n+1}^{\prime }\subset G_{n}^{\prime }$) for any $n\geq k_{0}$,
it follows that
\begin{equation*}
K\cap \lbrack A,B]=\cap _{n=k_{0}}^{\infty }G_{n}\text{ and }K\cap \lbrack
M,N]=\cap _{n=k_{0}}^{\infty }G_{n}^{\prime }.
\end{equation*}%
The continuity of $F$ yields that
\begin{equation*}
F(K\cap \lbrack A,B],K\cap \lbrack M,N])=\cap _{n=k_{0}}^{\infty
}F(G_{n},G_{n}^{\prime }).
\end{equation*}
In terms of the relation $G_{n+1}=\widetilde{G_{n}}$, $G_{n+1}^{\prime }=
\widetilde{G_{n}^{\prime }}$ and the condition in the lemma, it follows that
\begin{eqnarray*}
F(G_{n},G_{n}^{\prime }) &=&\cup _{1\leq i\leq t_{n}}\cup _{1\leq j\leq
t_{n}^{\prime }}F(I_{n,i},J_{n,j}) \\
&=&\cup _{1\leq i\leq t_{n}}\cup _{1\leq j\leq t_{n}^{\prime }}F(\widetilde{
I_{n,i}},\widetilde{J_{n,j}}) \\
&=&F(\cup _{1\leq i\leq t_{n}}\widetilde{I_{n,i}},\cup _{1\leq j\leq
t_{n}^{\prime }}\widetilde{J_{n,j}}) \\
&=&F(G_{n+1},G_{n+1}^{\prime }).
\end{eqnarray*}
Therefore, $F(K\cap \lbrack A,B],K\cap \lbrack
M,N])=F(G_{k_{0}},G_{k_{0}}^{\prime }).$
\end{proof}

The following two  lemmas are trivial. We shall use them frequently in the remaining paper. 
\begin{lemma}\label{fund}
Let $K$  be the attractors of the following IFS
 $$\{f_1(x)=\lambda  x, f_2(x)=\lambda x +c-\lambda,f_3(x)=\lambda x +1-\lambda\}, $$
 where  $f_1(I)\cap f_2(I)\neq \emptyset, (f_1(I)\cup f_2(I))\cap f_3(I)=\emptyset,$
and  $I=[0,1]$ is the convex hull of $K$. Then $\lambda\leq c\leq 2\lambda, c+\lambda<1.$
\end{lemma}
\begin{lemma}\label{frequentlyused}
The region satisfies the condition 
\begin{equation*}
\left\lbrace\begin{array}{cc}
                (1-\lambda)^2\leq c<1-\lambda\\
                \lambda\leq c\leq 2 \lambda\\
                \end{array}\right.
\end{equation*}
is the orange region in Figure 1.  Moreover, for any $(\lambda,c)$ in the  orange region, 
$$2c+\lambda-1\geq 0, 2\lambda+c-1\geq 0,$$ see the right picture in Figure 1. 
\end{lemma}
\begin{figure}
  \centering
    \includegraphics[width=180pt]{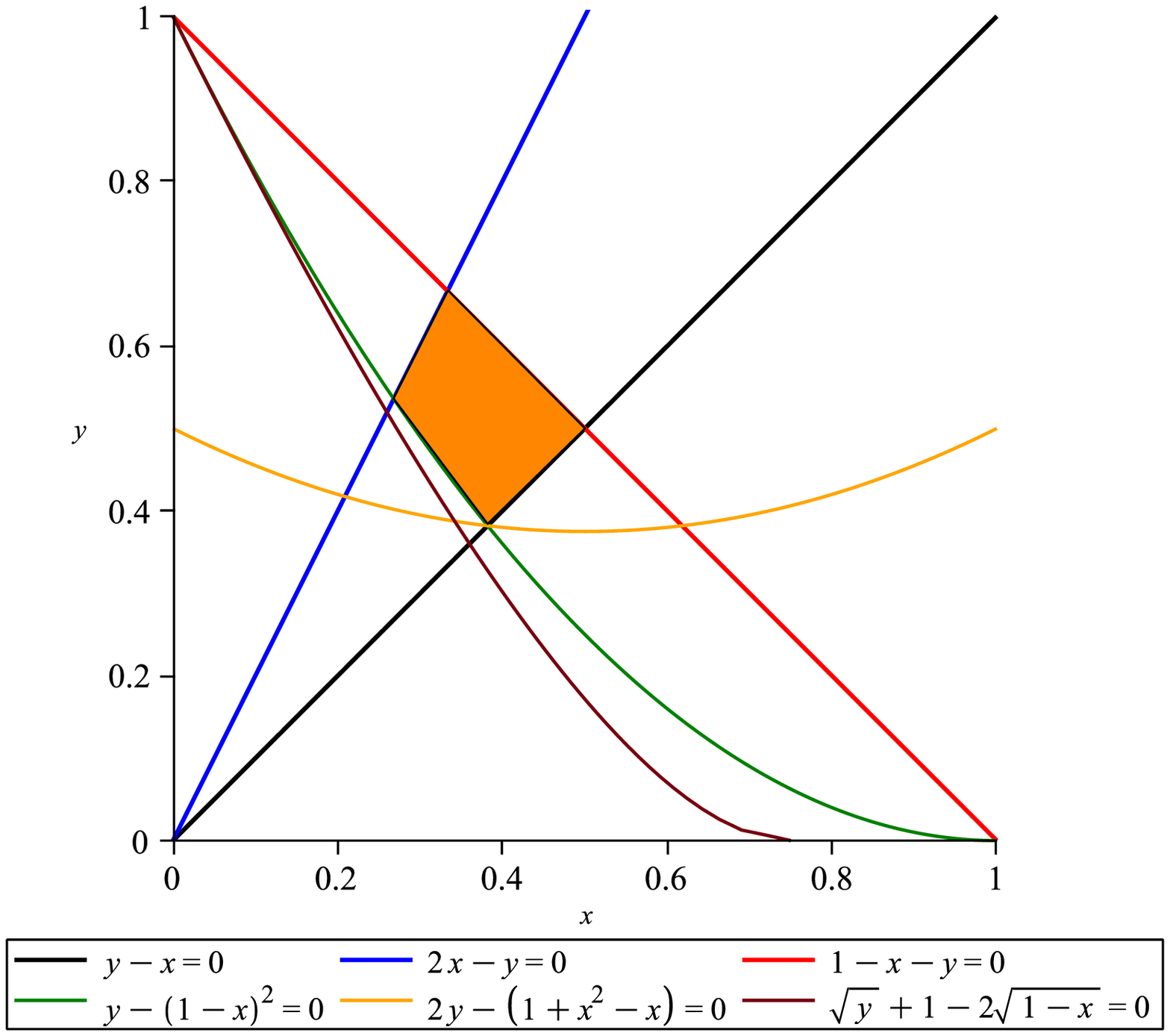}
  \includegraphics[width=175pt]{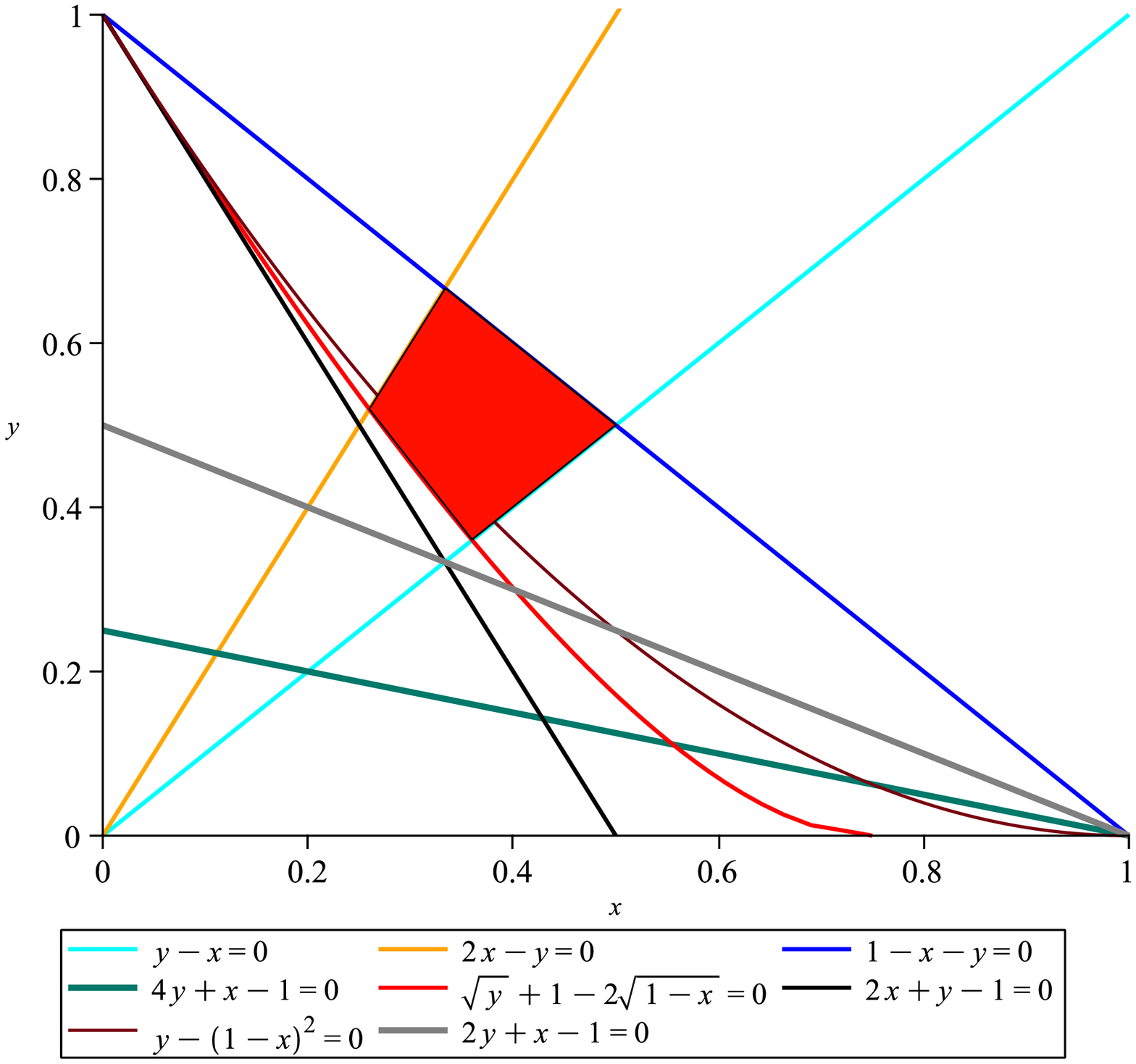}
   \caption{ }
\end{figure}
\begin{lemma}\label{lem1}
Let $I=[a,a+t], J=[b,b+t]$ be two basic intervals, where $b\geq a.$
Suppose  $a\geq 1-c-\lambda$, 
if $c\geq (1-\lambda)^2$,  then $f(I,J)=f(\tilde{I}, \tilde{J})$, where $f(x,y)=\dfrac{x}{y}$. 
\end{lemma}

\begin{proof}
Note that $\tilde{I}=[a,a+ct]\cup [a+(1-\lambda)t,a+t],\tilde{J}=[b,b+ct]\cup [b+(1-\lambda)t,b+t].$
Clearly, $$f(I,J)=\left[\dfrac{a}{b+t}, \dfrac{a+t}{b}\right],$$
$$f(\widetilde{I}, \widetilde{J})=J_1\cup J_2\cup J_3\cup J_4,$$
where 
\begin{eqnarray*}
J_1&=&\left[\dfrac{a}{b+t}, \dfrac{a+ct}{b+(1-\lambda)t}\right]=[e_1,h_1],   \\
J_2&=&\left[\dfrac{a}{b+ct}, \dfrac{a+ct}{b}\right]=[e_2, h_2],\\
J_3&=&\left[\dfrac{a+(1-\lambda)t}{b+t}, \dfrac{a+t}{b+(1-\lambda)t}\right]=[e_3,h_3],\\
J_4&=& \left[ \dfrac{a+(1-\lambda)t}{b+ct}, \dfrac{a+t}{b}\right]=[e_4,h_4].
\end{eqnarray*}
Since $b\geq a$, it follows that 
\begin{eqnarray*}
e_3-e_2 &=& \dfrac{t}{(b+ct)(b+t)}(-a(1-c)+(b+ct)(1-\lambda))\\ 
              &\geq& \dfrac{t}{(b+ct)(b+t)}(-a(1-c)+(b+ct)(1-c))\\
                      &=&\dfrac{t}{(b+ct)(b+t)}(b+ct-a)(1-c)\geq 0.
\end{eqnarray*}
Note that $f(I,J)=f(\widetilde{I}, \widetilde{J})$ if and only if 
\begin{equation*}
\left\lbrace\begin{array}{cc}
               h_1-e_2\geq 0\\
                h_2-e_3\geq0\\
                h_3-e_4 \geq0.\\
                \end{array}\right.
\end{equation*}
Now we prove these inequalities. 

\noindent\textbf{Case 1.}  
 \begin{eqnarray*}
h_1-e_2 &=&\dfrac{t}{(b+ct)(b+t-\lambda t)}(a\lambda-a+c^2t+ac+bc)\\ 
           &\geq& \dfrac{t}{(b+ct)(b+t-\lambda t)}(a\lambda-a+c^2t+ac+ac)\\
           &=& \dfrac{t}{(b+ct)(b+t-\lambda t)}(a(\lambda+2c-1)+c^2t).
\end{eqnarray*}
Therefore, we need to assume $\lambda+2c-1\geq 0$.

\noindent\textbf{Case 2.}  $h_2-e_3=\dfrac{t}{b(b+t)}(a+b(c+\lambda-1)+ct)\geq 0$
as  $b\geq a\geq 1-c-\lambda$. 

\noindent\textbf{Case 3.} 
 \begin{eqnarray*}
h_3-e_4 &=&\dfrac{a+t}{b+(1-\lambda)t}-\dfrac{a+(1-\lambda)t}{b+ct}\\ 
           &=& \dfrac{t}{(b+ct)(b+t-t\lambda)}(a\lambda-t-a+b\lambda+2t\lambda-t\lambda^2+ac+ct)\\
           &\geq& \dfrac{t}{(b+ct)(b+t-t\lambda)}(t(c-(1-\lambda)^2)+a(2\lambda+c-1)).
\end{eqnarray*}
If  $c-(1-\lambda)^2\geq 0$ and $2\lambda+c-1\geq0$, then $h_3-e_4\geq 0$.
By Lemmas \ref{fund}, \ref{frequentlyused} and the condition $c-(1-\lambda)^2\geq 0$, it follows that if $c-(1-\lambda)^2\geq 0$, then 
$2\lambda+c-1\geq0,\lambda+2c-1\geq 0$ . 
\end{proof}
\begin{lemma}\label{big}
If $\dfrac{3-\sqrt{5}}{2}\leq \lambda< 1 $, then $\dfrac{K}{K}=[0,\infty)$.
\end{lemma}
\begin{proof}
Note that $1-\lambda\geq 1-c-\lambda$,  by Lemma \ref{lem1} we may take  $I=J=[1-\lambda,1]$. 
Therefore, by Lemma \ref{lem} $$\dfrac{K}{K}\supset f([1-\lambda,1], [1-\lambda,1])=\left[1-\lambda, \dfrac{1}{1-\lambda}\right].$$
Since $\dfrac{3-\sqrt{5}}{2}\leq \lambda\leq 1/2 $, it follows that $\dfrac{\lambda}{1-\lambda}\geq 1-\lambda.$ Therefore, 
$$ [0,\infty)=\{0\}\cup  \cup_{k=-\infty}^{\infty}\lambda^k\left[1-\lambda, \frac{1}{1-\lambda}\right]\subset \dfrac{K}{K}\subset [0,\infty).$$
\end{proof}
\noindent Note that in the orange  region (see Figure 1), we have 
$$ c-\lambda^2\geq 1-c-\lambda.$$
In fact, $y=\dfrac{1}{2}(x^2-x+1)$ takes the minimun value at $\dfrac{3-\sqrt{5}}{2}$ on the interval $[0,1]$, i.e. 
$c_{min}\geq \dfrac{3-\sqrt{5}}{2}$ if $(\lambda, c)$ is in the orange region of Figure 1. 
Therefore, we can make use of Lemma \ref{lem1}.
Let $$I=J=J_1\cup J_2\cup J_3.$$
where
 \begin{eqnarray*}
J_1&=&\left[c-\lambda^2, c\right]\\
J_2&=& \left[ 1-\lambda, 1-\lambda+\lambda c\right]\\
J_3&=&\left[1-\lambda^2,1 \right].
\end{eqnarray*}
Then it  is easy to check that $f(I,J)=\cup_{i=1}^{9} L_i$, where
$$ L_{1}= \left[c-\lambda^2, \dfrac{c}{1-\lambda^2}\right], L_{2}= \left[\dfrac{c-\lambda^2}{ 1-\lambda+c\lambda}, \dfrac{ c}{1-\lambda}\right],L_{3}=\left[ 1-\lambda ,\frac{1-\lambda +\lambda c}{1-\lambda ^{2}}\right]$$
$$L_4=[*,*],  L_5=[*,*], L_6= 
\left[1-\lambda^2, \dfrac{1}{1-\lambda^2}\right]$$
$$L_7=\left[ \dfrac{1-\lambda ^{2}}{1-\lambda+c \lambda},\frac{1}{1-\lambda }\right], L_8= \left[\frac{1-\lambda }{c},\frac{1-\lambda +\lambda c}{c-\lambda ^{2}}\right] , L_9= \left[\frac{1-\lambda ^{2}}{c},
\frac{1}{c-\lambda ^{2}}\right].$$
We arrange $L_i=[i_l, i_r], 1\leq i\leq 9$ from left to right, where $``l",``r" $ denote the words left and right, respectively. 
Here $$L_4= \left[\dfrac{c-\lambda^2}{ c}, \dfrac{ c}{c-\lambda^2}\right],  L_5= \left[\frac{
1-\lambda }{1-\lambda +\lambda c},\frac{1-\lambda +\lambda c}{1-\lambda }\right],$$ provided that 
$\frac{
1-\lambda }{1-\lambda +\lambda c}\geq \dfrac{c-\lambda^2}{ c}$. 
If  $\frac{
1-\lambda }{1-\lambda +\lambda c}<\dfrac{c-\lambda^2}{ c}$, then 
$$L_4= \left[\frac{
1-\lambda }{1-\lambda +\lambda c},\frac{1-\lambda +\lambda c}{1-\lambda }\right],  L_5= \left[\dfrac{c-\lambda^2}{ c}, \dfrac{ c}{c-\lambda^2}\right].$$
The reason why  $i_l<(i+1)_{l}, i=1,2,3, 5,6,7,8$ is due  to following lemma. 
\begin{lemma}
Let $L_i=[i_l, i_r], 1\leq i\leq 9$  be the intervals defined as above. 
Then 
\begin{itemize}
\item[(1)]$2_l<3_l$ if and only if $c\leq \dfrac{\lambda^2+(1-\lambda)^2}{1-(1-\lambda)\lambda}$.
 \item[(2)]$\max\{4_l, 5_l\}<6_l$ if and only if $\max\left\{\dfrac{c-\lambda^2}{c}, \dfrac{1-\lambda}{1-\lambda+\lambda c}\right\}<1-\lambda^2$.
 \item[(3)]$7_l<8_l$ if and only if $\lambda+ c\leq 1$.
  \item[(4)] $\dfrac{1}{1-\lambda+c \lambda}>1, c\geq \lambda$.
  \end{itemize}
\end{lemma}
\begin{proof}
It suffices to show that $c\leq \dfrac{\lambda^2+(1-\lambda)^2}{1-(1-\lambda)\lambda}$.
Note that the above inequality is  equivalent to 
$$\lambda^2(c-1)\leq (1-\lambda)(1-c-\lambda).$$
The left side  is negative while the right   is positive.  
\end{proof}
\begin{lemma}\label{small}
Let $K$  be the attractors of the following IFS
 $$\{f_1(x)=\lambda  x, f_2(x)=\lambda x +c-\lambda,f_3(x)=\lambda x +1-\lambda\}, $$
 where  $f_1(I)\cap f_2(I)\neq \emptyset, (f_1(I)\cup f_2(I))\cap f_3(I)=\emptyset,$
and  $I=[0,1]$ is the convex hull of $K$.  If $c\geq (1-\lambda)^2$ and $0<\lambda\leq \dfrac{3-\sqrt{5}}{2}$, then 
$$\dfrac{K}{K}\supset \left[c-\lambda ^{2},\frac{1}{c-\lambda ^{2}}\right].$$
\end{lemma}
\begin{lemma}\label{case1}
Suppose $(\lambda,c)$ satisfies the conditions in Lemma \ref{small}.
If $$\frac{
1-\lambda }{1-\lambda +\lambda c}\geq \dfrac{c-\lambda^2}{ c},$$ then  $L_i\cap L_{i+1}\neq \emptyset, 1\leq i\leq 8$.
\end{lemma}
\begin{proof}
Note that $$L_4= \left[\dfrac{c-\lambda^2}{ c}, \dfrac{ c}{c-\lambda^2}\right],  L_5= \left[\frac{
1-\lambda }{1-\lambda +\lambda c},\frac{1-\lambda +\lambda c}{1-\lambda }\right],$$ if 
$\frac{
1-\lambda }{1-\lambda +\lambda c}\geq \dfrac{c-\lambda^2}{ c}$. 

\noindent\textbf{Case 1.} 
$$1_r-2_l=\frac{c}{1-\lambda ^{2}}-\frac{c-\lambda ^{2}}{
1-\lambda +\lambda c}=\allowbreak -\frac{\lambda }{\left( \lambda
^{2}-1\right) \left( c\lambda -\lambda +1\right) }\left( c^{2}+c\lambda
-c-\lambda ^{3}+\lambda \right)\geq 0.$$
Note that $c<1-\lambda<1-\lambda^2$ and $\lambda\geq (1-c-\lambda)$ (Lemma \ref{frequentlyused}), therefore, 
$$\lambda(1-\lambda^2)\geq c (1-c-\lambda)\Leftrightarrow c^{2}+c\lambda
-c-\lambda ^{3}+\lambda\geq 0.$$

\noindent\textbf{Case 2. }
$$2_r-3_l=\frac{c}{1-\lambda }-(1-\lambda )=\dfrac{1}{1-\lambda}(-\lambda^2+2\lambda+c-1)=\dfrac{1}{1-\lambda}(c-(1-\lambda)^2)\geq 0.$$

\noindent\textbf{Case 3. }
 \begin{eqnarray*}
 3_r-4_l=
\frac{1-\lambda +\lambda c}{1-\lambda ^{2}}-\frac{c-\lambda^2 }{
 c}= -\frac{1}{c}\frac{\lambda }{\lambda ^{2}-1}%
\left( c^{2}+c\lambda -c-\lambda ^{3}+\lambda \right) \geq 0.\\
\end{eqnarray*}
\noindent\textbf{Case 4.}

 \begin{eqnarray*}
 4_r-5_l=\frac{ c}{c-\lambda^2}-\dfrac{1-\lambda
}{1-\lambda+\lambda c}= \lambda \frac{c^{2}-\lambda ^{2}+\lambda }{c^{2}\lambda
-c\lambda ^{3}-c\lambda +c+\lambda ^{3}-\lambda ^{2}}\geq 0.\\
\end{eqnarray*}

\noindent\textbf{Case 5.}

 \begin{eqnarray*}
  5_r-6_l=
\frac{1-\lambda +\lambda c}{1-\lambda }-(1-\lambda
^{2})=-\frac{\lambda }{\lambda -1}\left( -\lambda ^{2}+\lambda
+c\right) \geq 0.
\end{eqnarray*}

\noindent\textbf{Case 6.}

 \begin{eqnarray*}
  6_r-7_l=
\frac{1}{1-\lambda ^{2}}-\frac{1-\lambda ^{2}}{1-\lambda
+\lambda c}=\allowbreak -\frac{\lambda }{\left( \lambda ^{2}-1\right) \left(
c\lambda -\lambda +1\right) }\left( -\lambda ^{3}+2\lambda +c-1\right) \geq 0.
\end{eqnarray*}
Here 
$ -\lambda ^{3}+2\lambda +c-1\geq  -\lambda ^{2}+2\lambda +c-1=c-(1-\lambda)^2$. 

\noindent\textbf{Case 7.}

 \begin{eqnarray*}
  7_r-8_l=
\frac{1}{1-\lambda }-\frac{1-\lambda }{c}=\allowbreak -\frac{1}{%
c\left( \lambda -1\right) }\left( -\lambda ^{2}+2\lambda +c-1\right) \geq 0.
\end{eqnarray*}

\noindent\textbf{Case 8.}

 \begin{eqnarray*}
  8_r-9_l=
\frac{1-\lambda +\lambda c}{c-\lambda ^{2}}-\frac{1-\lambda ^{2}%
}{c}=\allowbreak \frac{1}{c}\frac{\lambda }{c-\lambda ^{2}}\left(
c^{2}+c\lambda -c-\lambda ^{3}+\lambda \right) \geq 0.
\end{eqnarray*}
\end{proof}
\begin{lemma}\label{case2}
Suppose $(\lambda,c)$ satisfies the conditions in Lemma \ref{small}.
If $$\frac{
1-\lambda }{1-\lambda +\lambda c}< \dfrac{c-\lambda^2}{ c},$$ then  $L_i\cap L_{i+1}\neq \emptyset, 1\leq i\leq 8$.
\end{lemma}
\begin{proof}
If  $\frac{
1-\lambda }{1-\lambda +\lambda c}<\dfrac{c-\lambda^2}{ c}$, then 
$$L_4= \left[\frac{
1-\lambda }{1-\lambda +\lambda c},\frac{1-\lambda +\lambda c}{1-\lambda }\right],  L_5= \left[\dfrac{c-\lambda^2}{ c}, \dfrac{ c}{c-\lambda^2}\right].$$
With a similar discussion of  Lemma \ref{case1}, it suffices to prove  the following three cases.

\noindent\textbf{Case 1. }
 \begin{eqnarray*}
 3_r-4_l&=&
\frac{1-\lambda +\lambda c}{1-\lambda ^{2}}-\frac{1-\lambda }{1-\lambda
+\lambda c}\\
&=&\allowbreak -\frac{\lambda }{\left( \lambda ^{2}-1\right) \left(
c\lambda -\lambda +1\right) }\left( c^{2}\lambda -2c\lambda +2c-\lambda
^{2}+2\lambda -1\right) \geq 0\allowbreak.\\
\end{eqnarray*}
Here $$c^{2}\lambda -2c\lambda +2c-\lambda
^{2}+2\lambda -1\geq 0$$
is equivalent to $c(c\lambda-2\lambda+2)\geq (1-\lambda)^2$. Since 
$c\geq (1-\lambda)^2$, it suffices to prove that 
$$c\lambda-2\lambda+2\geq 1.$$ By the assumption $0<\lambda\leq \dfrac{3-\sqrt{5}}{2}$, the above inequality holds.

\noindent\textbf{Case 2.}

 \begin{eqnarray*}
 4_r-5_l=\frac{1-\lambda +\lambda c}{1-\lambda }-\frac{c-\lambda ^{2}}{c%
}=\allowbreak -\frac{1}{c}\frac{\lambda }{\lambda -1}\left( c^{2}-\lambda
^{2}+\lambda \right) \geq 0.
\end{eqnarray*}

\noindent\textbf{Case 3.}

 \begin{eqnarray*}
 5_r-6_l=
\frac{c}{c-\lambda ^{2}}-(1-\lambda ^{2})=\allowbreak \frac{\lambda ^{2}%
}{c-\lambda ^{2}}\left( -\lambda ^{2}+c+1\right) \geq 0.
\end{eqnarray*}

\end{proof}

\begin{proof}[Proof of Lemma \ref{small}]
Lemma \ref{small} follows from 
 Lemmas \ref{case1},  \ref{case2}, \ref{lem1}, and  \ref{lem}. 
\end{proof}
\begin{theorem}
If $c\geq (1-\lambda)^2$, then 
$\dfrac{K}{K}=[0,\infty)$. 
\end{theorem}
\begin{proof}
If $ \dfrac{3-\sqrt{5}}{2}\leq \lambda<1$, then by Lemma \ref{big},  $\dfrac{K}{K}=[0,\infty)$. 
If $0<\lambda<\dfrac{3-\sqrt{5}}{2}$ and $c\geq (1-\lambda)^2$, in terms of Lemma \ref{small}
we have $\dfrac{K}{K}\supset [c-\lambda ^{2},\frac{1}{c-\lambda ^{2}}]$. We prove
$$\frac{\lambda }{c-\lambda ^{2}}-(c-\lambda ^{2})\geq 0.$$ 
Recall $c\leq 2\lambda$, if we can show  $2\lambda\leq \lambda^2+\sqrt{\lambda}$, then we prove 
$$\frac{\lambda }{c-\lambda ^{2}}-(c-\lambda ^{2})\geq 0.$$ 
However, $2\lambda\leq \lambda^2+\sqrt{\lambda}$ is equivalent to 
$$\lambda^3-4\lambda^2+4\lambda-1=(\lambda-1)(\lambda^2-3\lambda+1)\leq 0,$$
 which is a consequence of $0<\lambda<\dfrac{3-\sqrt{5}}{2}$. 
Therefore,  in terms of $$\frac{\lambda }{c-\lambda ^{2}}-(c-\lambda ^{2})\geq 0,$$ we conclude that 
$$ [0,\infty)=\{0\}\cup  \cup_{k=-\infty}^{\infty}\lambda^k[c-\lambda ^{2},\frac{1}{c-\lambda ^{2}}]\subset \dfrac{K}{K}\subset [0,\infty).$$
\end{proof}
\begin{lemma}\label{lem+}
Let $I=[a,a+t], J=[b,b+t]$ be two basic intervals, where $b\geq a.$
If $c\geq (1-\lambda)^2$,  then $f(I,J)=f(\tilde{I}, \tilde{J})$, where $f(x,y)=x+y$. 
\end{lemma}
\begin{proof}
Note that $\tilde{I}=[a,a+ct]\cup [a+(1-\lambda)t,a+t],\tilde{J}=[b,b+ct]\cup [b+(1-\lambda)t,b+t].$
Clearly, $$f(\widetilde{I},\widetilde{J})=J_1\cup J_2\cup J_3, $$
where 
\begin{eqnarray*}
J_1&=&[a+b,a+b+2ct]=[1_l,1_r],   \\
J_2&=&\lbrack a+b+(1-\lambda
)t,a+ct+b+t]=[2_l,2_r],\\
J_3&=&\lbrack a+b+2(1-\lambda )t,a+b+2t]=[3_l,3_r].\\
\end{eqnarray*}
By virtue of   Lemma \ref{frequentlyused}, it follows that 
$$1_{r}-2_{l}=a+b+2ct-(a+b+(1-\lambda )t)=\allowbreak t\left( 2c+\lambda
-1\right)\geq 0, $$
$$2_{r}-3_{l}=a+ct+b+t-(a+b+2(1-\lambda )t)=\allowbreak t\left( c+2\lambda
-1\right)\geq 0. $$
\end{proof}
\begin{lemma}\label{+}
If $c\geq (1-\lambda)^2$, then $K+K=[0,2]$. 
\end{lemma}
\begin{proof}
By Lemmas \ref{lem} and  \ref{lem+}.
Take
$I=J=[c-\lambda ,c]\cup \lbrack 1-\lambda ,1]$. 
Therefore, for $f(x,y)=x+y$, we have 
$$f(I,J)=[2(c-\lambda ),2c]\cup \lbrack
c+1-2\lambda ,1+c]\cup \lbrack 2(1-\lambda ),2].$$
By Lemma \ref{frequentlyused}, we conculde that 
$$2c-(c+1-2\lambda )=\allowbreak c+2\lambda -1\geq 0, 1+c-(2(1-\lambda ))=\allowbreak c+2\lambda -1\geq 0.$$
Since $c\leq 2\lambda$, it follows that 
$$ [0,2]=\{0\}\cup  \cup_{k=0}^{\infty}\lambda^k[2(c-\lambda ),2]\subset \dfrac{K}{K}\subset [0,2].$$
\end{proof}
\begin{remark}
We may give another proof of this result. Note that $K+K$ is a self-similar set, namely, 
$$K+K=\left\{\sum_{i=1}^{\infty}\dfrac{a_i}{q^i}:a_i\in\{0, d_1, d_2, 2d_1, d_1+d_2,2d_2\}\right\},$$
where $d_1=\dfrac{c}{\lambda}-1, d_2=\dfrac{1}{\lambda}-1, q=\dfrac{1}{\lambda}.$
The IFS of $K+K$ is $\{g_i\}_{i=1}^{6}$, where 
$$g_1(x)=\dfrac{x}{q}, g_2(x)=\dfrac{x+d_1}{q},g_3(x)=\dfrac{x+d_2}{q},$$
 $$g_4(x)=\dfrac{x+2d_1}{q},\\
g_5(x)=\dfrac{x+d_1+d_2}{q},g_6(x)=\dfrac{x+2d_2}{q}.$$
Let $E=[0,2]$. It  is easy to check that if $2\lambda+c-1\geq 0$ and $2c+\lambda-1\geq 0$, then 
$$\cup_{i=1}^{6}g_i(E)=[0,2].$$
Therefore, if $c\geq (1-\lambda)^2$ (Lemma \ref{frequentlyused}), then $K+K=[0,2].$
\end{remark}
We shall use this idea to prove the following result. 
\begin{lemma}\label{-}
If $c\geq (1-\lambda)^2$, then $K-K=[-1,1]$. 
\end{lemma}
\begin{proof}
First, 
$$K-K=\left\{\sum_{i=1}^{\infty}\dfrac{a_i}{q^i}:a_i\in\{-d_2, -d_1, d_1-d_2, 0, d_2-d_1, d_1, d_2\}\right\},$$
where $d_1=\dfrac{c}{\lambda}-1, d_2=\dfrac{1}{\lambda}-1, q=\dfrac{1}{\lambda}.$
The IFS of $K+K$ is $\{h_i\}_{i=1}^{7}$, where 
$$h_1(x)=\dfrac{x-d_2}{q}, h_2(x)=\dfrac{x-d_1}{q},h_3(x)=\dfrac{x+d_1-d_2}{q},$$
 $$h_4(x)=\dfrac{x}{q},
h_5(x)=\dfrac{x+d_2-d_1}{q},h_6(x)=\dfrac{x+d_1}{q}, h_7(x)=\dfrac{x+d_2}{q}.$$
Let $M=[-1,1]$.
It is not difficult to check that if $2c+\lambda-1\geq 0 $ and $2\lambda+c-1\geq 0 $, then 
$$\cup_{i=1}^{7}h_i(M)=[-1,1].$$
In other words, if $c\geq (1-\lambda)^2$ (which implies $2c+\lambda-1\geq 0 $ and $2\lambda+c-1\geq 0 $), then 
$K-K=[-1,1]$. 
\end{proof}
\begin{proof}[\textbf{Proof of Corollary \ref{Cor}}]
First, in \cite{XiKan1}, Tian et al. proved $$K\cdot K=[0,1]\mbox{ if and only if }c\geq (1-\lambda)^2.$$ Therefore, $(2)\Rightarrow(1)\Leftrightarrow (3)$. By Lemmas \ref{+}, \ref{-}, Theorems \ref{Main} and \ref{Main1}, 
$(3)\Rightarrow(2)$, we are done. 
\end{proof}
\section{The necessary and sufficient condition for $\sqrt{K}+\sqrt{K}=[0,2]$}
In this section, we prove that $\sqrt{K}+\sqrt{K}=[0,2]$ if and only if $\sqrt{c}+1\geq 2\sqrt{1-\lambda}.$
The necessary condition is trivial as $\sqrt{K}\subset [0,\sqrt{c}]\cup [\sqrt{1-\lambda}, 1]$, and therefore, $$\sqrt{K}+\sqrt{K}\subset [0,2\sqrt{c}] \cup [\sqrt{1-\lambda}, 1+\sqrt{c}] \cup [2\sqrt{1-\lambda}, 2].$$
If $\sqrt{K}+\sqrt{K}=[0,2]$, then $2\sqrt{c}\geq \sqrt{1-\lambda}, 1+\sqrt{c}\geq 2\sqrt{1-\lambda}.$ However, by Figure 1 and the definition of $K$ (we may  also use Lemma \ref{useful}), $1+\sqrt{c}\geq 2\sqrt{1-\lambda}$ implies $2\sqrt{c}\geq \sqrt{1-\lambda}$. 

Now we prove that if $\sqrt{c}+1\geq 2\sqrt{1-\lambda},$ then $\sqrt{K}+\sqrt{K}=[0,2]$. In other words, when $(\lambda, x)$ is  in the red region in Figure 1, then we must have $\sqrt{K}+\sqrt{K}=[0,2]$.
We give a simple outline of this proof.  Firstly, we prove that if $(\lambda, c)$ in the blue region of Figure 2, we have $\sqrt{K}+\sqrt{K}=[0,2]$, see Theorem \ref{blue}. Secondly, if $(\lambda, c)$ in the orange region of Figure 3, then 
$\sqrt{K}+\sqrt{K}=[0,2]$, see Theorem \ref{orange}.  Note that the union of orange region of Figure 3 and the blue region of Figure 2 is exactly the red region of Figure 1. Therefore, we prove the desired result. 

 The following lemma will be used intensively in this section. 
 \begin{lemma}\label{useful}
 If $(\lambda, c)$ satisfies the following conditions,  
\begin{equation*}
\left\lbrace\begin{array}{cc}
                     c<1-\lambda\\
                \lambda\leq c\leq 2 \lambda\\
                \sqrt{c}+1\geq 2\sqrt{1-\lambda},
                \end{array}\right.
\end{equation*}
 then 
 \begin{equation*}
\left\lbrace\begin{array}{cc}
                     c\geq (1-\lambda)^2\\
                2\lambda+c-1\geq 0\\
             \lambda+2c-1\geq 0\\
             4c+\lambda\geq 1.
                \end{array}\right.
\end{equation*}
 \end{lemma}
 \begin{proof}
 The proof is due to Figure 1.
 \end{proof}
 The following lemma guarantees  that Lemma \ref{lem} can be used when we calculate $\sqrt{K}+\sqrt{K}.$
\begin{lemma}\label{sqrt1}
Let $I=[a,a+t],J=[b,b+t]$ be two basic intervals.
Suppose that
\begin{eqnarray*}
b &\geq &a\geq (1-c-\lambda )^{2} \\
8a(2\lambda +c-1) &\geq &t(3-4\lambda -4\lambda c-c^{2}-2c).
\end{eqnarray*}
\noindent
Then
\begin{equation*}
f(I,J)=f(\tilde{I},\tilde{J}),
\end{equation*}
\noindent
where $f(x,y)=\sqrt{x}+\sqrt{y}.$
\end{lemma}
\begin{proof}
By the definition of $I,J,$ we clearly have 
\begin{eqnarray*}
\tilde{I} &=&[a,a+ct]\cup \lbrack a+(1-\lambda )t,a+t] \\
\tilde{J} &=&[b,b+ct]\cup \lbrack b+(1-\lambda )t,b+t].
\end{eqnarray*}
\noindent
Therefore, we have that
\begin{equation*}
f(\widetilde{I},\widetilde{J})=J_{1}\cup J_{2}\cup J_{3}\cup J_{4},
\end{equation*}
\noindent
where 
\begin{eqnarray*}
J_{1} &=&[\sqrt{a}+\sqrt{b},\sqrt{a+ct}+\sqrt{b+ct}]=[h_{1,}e_{1}], \\
J_{2} &=&[\sqrt{a}+\sqrt{b+(1-\lambda )t},\sqrt{a+ct}+\sqrt{b+t}
]=[h_{2,}e_{2}], \\
J_{3} &=&[\sqrt{a+(1-\lambda )t}+\sqrt{b},\sqrt{a+t}+\sqrt{b+ct}
]=[h_{3,}e_{3}], \\
J_{4} &=&[\sqrt{a+(1-\lambda )t}+\sqrt{b+(1-\lambda )t},\sqrt{a+t}+\sqrt{b+t}
]=[h_{4,}e_{4}].
\end{eqnarray*}
\noindent
In terms of the condition $b\geq a,$ it follows that%
\begin{eqnarray*}
&&\sqrt{a+(1-\lambda )t}+\sqrt{b}-(\sqrt{a}+\sqrt{b+(1-\lambda )t}) \\
&=&\sqrt{a+t-t\lambda }-\sqrt{b+t-t\lambda }-\sqrt{a}+\sqrt{b} \\
&=&\frac{t-t\lambda }{\sqrt{a+t-t\lambda }+\sqrt{a}}-\frac{t-t\lambda }{%
\sqrt{b+t-t\lambda }+\sqrt{b}} \\
&\geq &\frac{t-t\lambda }{\sqrt{b+t-t\lambda }+\sqrt{b}}-\frac{t-t\lambda }{%
\sqrt{b+t-t\lambda }+\sqrt{b}}=0
\end{eqnarray*}
Therefore, 
in order to prove 
\begin{equation*}
f(I,J)=f(\tilde{I},\tilde{J}),
\end{equation*}
it suffices to prove that
\begin{equation*}
\left\lbrace\begin{array}{cc}
               e_1-h_2\geq 0\\
                e_2-h_3\geq0\\
                e_3-h_4 \geq0.\\
                \end{array}\right.
\end{equation*}
Case $(1)$
\begin{eqnarray*}
e_{1}-h_{2} &=&\sqrt{a+ct}+\sqrt{b+ct}-(\sqrt{a}+\sqrt{b+(1-\lambda )t}) \\
&=&\allowbreak \sqrt{a+ct}-\sqrt{a}+\sqrt{b+ct}-\sqrt{b+t-t\lambda } \\
&=&\frac{ct}{\sqrt{a+ct}+\sqrt{a}}-\frac{(1-c-\lambda )t}{\sqrt{b+t-t\lambda 
}+\sqrt{b+ct}} \\
&\geq &\frac{ct}{\sqrt{b+t-t\lambda }+\sqrt{b+ct}}-\frac{(1-c-\lambda )t}{%
\sqrt{b+t-t\lambda }+\sqrt{b+ct}}=\frac{(2c+\lambda -1)t}{\sqrt{b+t-t\lambda 
}+\sqrt{b+ct}}\geq 0
\end{eqnarray*}
Case $(2)$
\begin{eqnarray*}
e_{2}-h_{3} &=&\sqrt{a+ct}+\sqrt{b+t}-(\sqrt{a+(1-\lambda )t}+\sqrt{b}) \\
&=&\sqrt{b+t}-\sqrt{a+t-t\lambda }+\sqrt{a+ct}-\sqrt{b} \\
&=&\frac{t}{\sqrt{b+t}+\sqrt{b}}-\frac{(1-c-\lambda )t}{\sqrt{a+t-t\lambda }+%
\sqrt{a+ct}} \\
&\geq &\frac{t}{2\sqrt{b+t}}-\frac{(1-c-\lambda )t}{2\sqrt{a+ct}}\geq 0
\end{eqnarray*}
Note that
\begin{eqnarray*}
\frac{t}{2\sqrt{b+t}}-\frac{(1-c-\lambda )t}{2\sqrt{a+ct}} &\geq
&0\Leftrightarrow \frac{a+ct}{b+t}\geq (1-c-\lambda )^{2}\Leftrightarrow \\
a &\geq &(1-c-\lambda )^{2}+t(1-c-\lambda )^{2}-ct
\end{eqnarray*}
However,
$t(1-c-\lambda )^{2}-ct\leq t(1-c-\lambda )-ct=t(1-2c-\lambda )\leq 0.$%
Therefore, if
\begin{equation*}
a\geq (1-c-\lambda )^{2},
\end{equation*}
then $e_{2}-h_{3}\geq 0.$

\noindent Case $(3)$ 

\noindent If we want to prove
\begin{equation*}
e_{3}-h_{4}=\sqrt{a+t}+\sqrt{b+ct}-(\sqrt{a+(1-\lambda )t}+\sqrt{%
b+(1-\lambda )t})\geq 0
\end{equation*}
it suffices to prove
\begin{equation*}
(\sqrt{a+t}+\sqrt{b+ct})^{2}\geq (\sqrt{a+(1-\lambda )t}+\sqrt{b+(1-\lambda
)t})^{2}
\end{equation*}
 Equivalently, we need to prove
\begin{equation*}
\allowbreak t\left( c+2\lambda -1\right) +2\sqrt{(a+t)(b+ct)}\geq 2\sqrt{%
(a+(1-\lambda )t)(b+(1-\lambda )t)}
\end{equation*}

or

\begin{equation*}
\allowbreak (t\left( c+2\lambda -1\right) +2\sqrt{(a+t)(b+ct)})^{2}\geq (2%
\sqrt{(a+(1-\lambda )t)(b+(1-\lambda )t)})^{2},
\end{equation*}%
namely,
$$(t\left( c+2\lambda -1\right) )^{2}+4(a+t)(b+ct)+4t\left( c+2\lambda
-1\right) \sqrt{(a+t)(b+ct)}\geq 4(a+(1-\lambda )t)(b+(1-\lambda )t)$$
or
\begin{equation*}
\left( 4a\lambda -3t-4a+4b\lambda +4t\lambda +c^{2}t+4ac+2ct+4ct\lambda
\right) \allowbreak +4\left( c+2\lambda -1\right) \sqrt{(a+t)(b+ct)}\geq 0.
\end{equation*}
Therefore,
\begin{equation*}
(4\lambda +c^{2}+2c+4c\lambda -3)t+8\left( c+2\lambda -1\right) a\geq 0
\end{equation*}
implies 
\begin{equation*}
\left( 4a\lambda -3t-4a+4b\lambda +4t\lambda +c^{2}t+4ac+2ct+4ct\lambda
\right) \allowbreak +4\left( c+2\lambda -1\right) \sqrt{(a+t)(b+ct)}\geq 0.
\end{equation*}
\end{proof}
With the help of computer, we obtain the following lemma. 
\begin{lemma}\label{blue1}
In the blue   region (Figure 2), we have 
\begin{equation*}
\left\lbrace\begin{array}{cc}
                c-\lambda \geq (1-c-\lambda )^{2},\\
               8(c-\lambda )(2\lambda +c-1)\geq \lambda (3-4\lambda -4\lambda
c-c^{2}-2c).
                \end{array}\right.
\end{equation*}
\end{lemma}
\begin{figure}
  \centering
  \includegraphics[width=250pt]{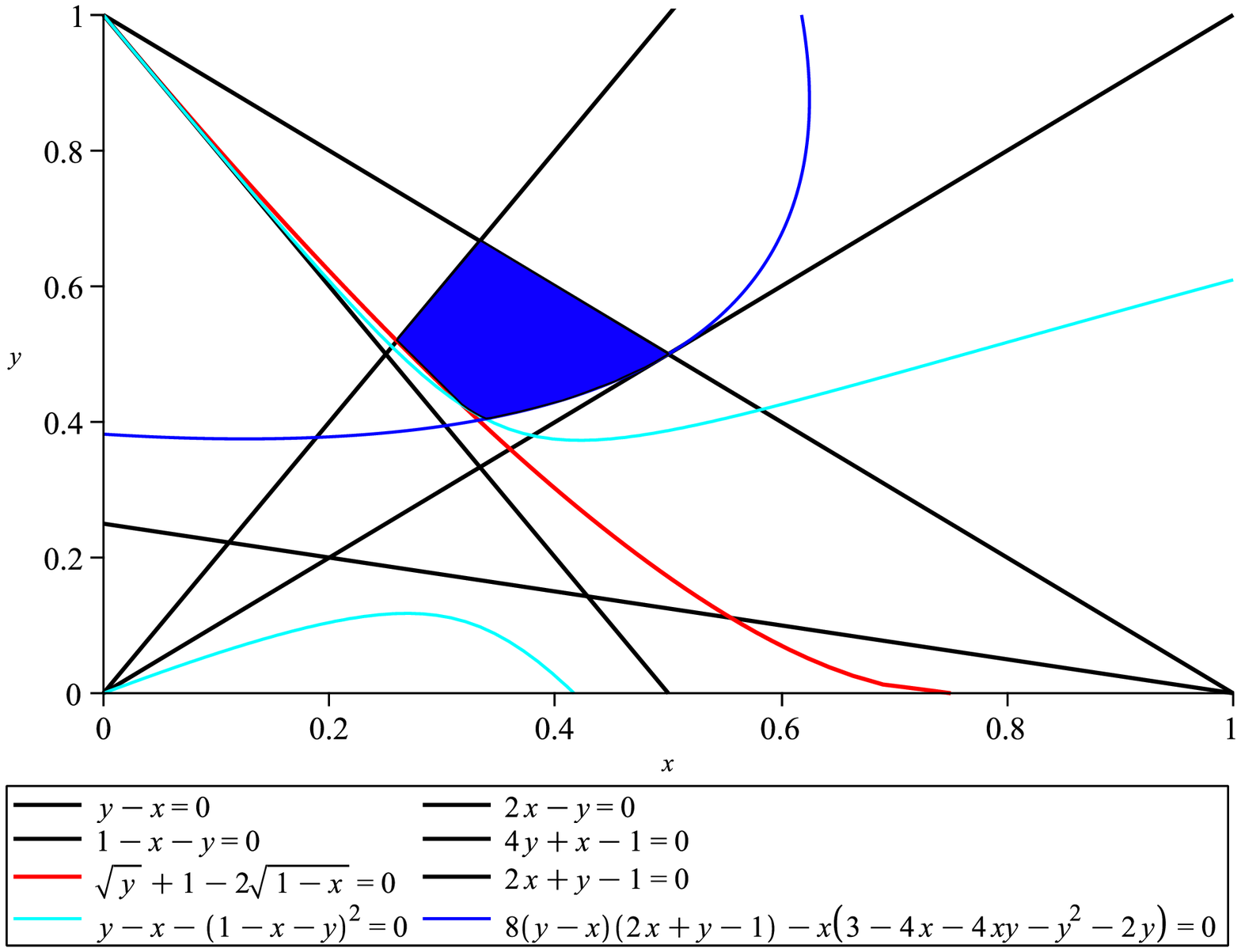}
   \caption{ }
\end{figure}
\begin{theorem}\label{blue}
In the blue region (Figure 2), we have $\sqrt{K}+\sqrt{K}=[0,2]$.
\end{theorem}
\begin{proof}
Note that the blue region is generated by the following inequalities
\begin{equation*}
\left\lbrace\begin{array}{cc}
\lambda+ c\leq 1\\
\lambda\leq c\leq 2\lambda\\
1+\sqrt{c}-2\sqrt{1-\lambda }\geq 0\\
                c-\lambda \geq (1-c-\lambda )^{2},\\
               8(c-\lambda )(2\lambda +c-1)\geq \lambda (3-4\lambda -4\lambda
c-c^{2}-2c).
                \end{array}\right.
\end{equation*}
In this blue region, we have that $c+2\lambda -1\geq 0$ (Lemma \ref{useful}).
Therefore, by Lemmas \ref{blue1}, \ref{sqrt1} and \ref{lem}, we let $$I=J=[c-\lambda ,c]\cup \lbrack 1-\lambda ,1],$$
and therefore, 
$$\sqrt{K}+\sqrt{K}\supset f(I,J)=[2\sqrt{c-\lambda },2\sqrt{c}]\cup
\lbrack \sqrt{1-\lambda }+\sqrt{c-\lambda },1+\sqrt{c}]\cup \lbrack
2\sqrt{1-\lambda },2].$$

\noindent Note that 
\begin{eqnarray*}
2\sqrt{c}-(\sqrt{1-\lambda }+\sqrt{c-\lambda }) &=&\frac{\lambda }{\sqrt{c}+\sqrt{c-\lambda }}+\frac{c+\lambda -1}{\sqrt{c}+
\sqrt{1-\lambda }}\\ &\geq &\frac{c+\lambda -1+\lambda }{\sqrt{c}+\sqrt{
c-\lambda }}\geq 0,
\end{eqnarray*}
and 
$$1+\sqrt{c}-2\sqrt{1-\lambda }\geq 0.$$
Therefore, 
$\sqrt{K}+\sqrt{K}\supset f(I,J)=[2\sqrt{c-\lambda },2]$, and 
$$[0,2]\subset \{0\}\cup  \cup_{k=0}^{\infty}(\sqrt{\lambda})^k[2\sqrt{c-\lambda },2]\subset \sqrt{K}+\sqrt{K}\subset [0,2],$$
where we use the inequality $2\sqrt{\lambda}>2\sqrt{c-\lambda}$ ($c\leq 2\lambda$). 
\end{proof}
Therefore, we only need to prove that  if $(\lambda, c)$ is in the  orange region of Figure 3, then $$\sqrt{K}+\sqrt{K}=[0,2].$$
\begin{lemma}\label{y1}
In the  orange  region (Figure 3), we have 
\begin{equation*}
\left\lbrace\begin{array}{cc}
                \lambda -\lambda ^{2}\geq (1-c-\lambda
)^{2},\\
              8(\lambda -\lambda ^{2})(2\lambda +c-1)\geq \lambda ^{2}(3-4\lambda
-4\lambda c-c^{2}-2c)
                \end{array}\right.
\end{equation*}
\end{lemma}
\begin{figure}
  \centering
  \includegraphics[width=300pt]{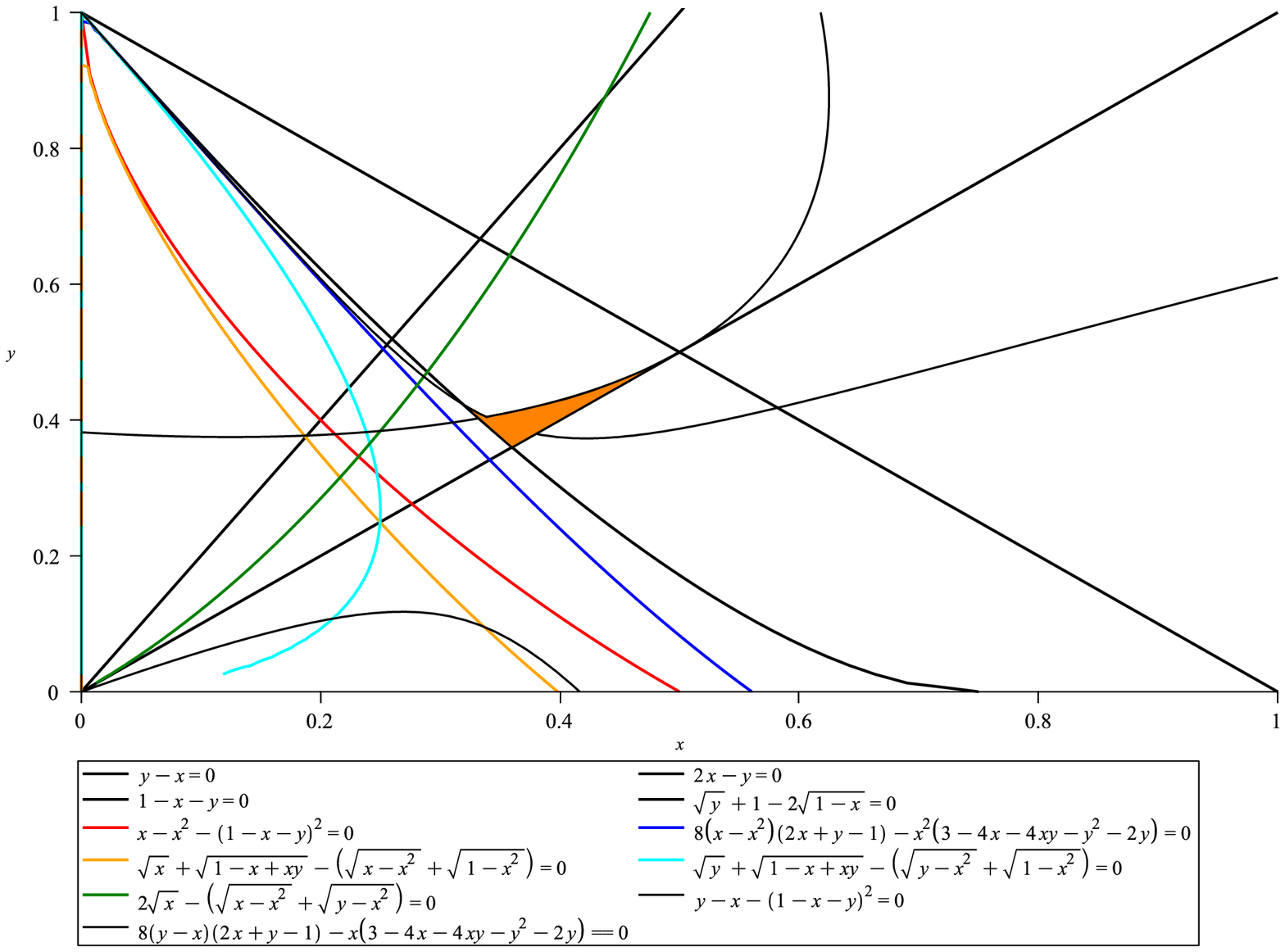}
   \caption{ }
\end{figure}
\begin{proof}
Note that  the points in the orange region of Figure 3 satisfy
\begin{equation*}
\left\lbrace\begin{array}{cc}
                x -x ^{2}\geq (1-y-x
)^{2},\\
              8(x -x^{2})(2x +y-1)\geq x^{2}(3-4x
-4x y-y^{2}-2y).
                \end{array}\right.
\end{equation*}
\end{proof}
\begin{theorem}\label{orange}
Suppose that  $(\lambda, c)$ is in the orange region of Figure 3. Then 
$$\sqrt{K}+\sqrt{K}=[0,2].$$
\end{theorem}
The proof of this result is a little complicated. We outline the proof. First, we prove Lemma \ref{gap6}, namely 
$$\sqrt{K}+\sqrt{K}\supset [\sqrt{\lambda-\lambda^2}+\sqrt{c-\lambda^2},\sqrt{\lambda}+\sqrt{c}]\cup \lbrack \sqrt{\lambda -\lambda ^{2}}+\sqrt{1-\lambda },\sqrt{\lambda }+1] \cup \lbrack \sqrt{c-\lambda
^{2}}+\sqrt{1-\lambda },2].$$
There are still two gaps, i.e.
$$[\sqrt{\lambda}+\sqrt{c}, \sqrt{\lambda-\lambda^2}+\sqrt{1-\lambda}], [\sqrt{\lambda}+1, \sqrt{c-\lambda^2}+\sqrt{1-\lambda}].$$  We need to find other intervals,  which are subsets of $\sqrt{K}+\sqrt{K}$, such that these intervals  cover the above gaps, see Lemmas \ref{gap5} and \ref{gap3}. Therefore, 
$$\sqrt{K}+\sqrt{K}\supset [\sqrt{\lambda-\lambda^2}+\sqrt{c-\lambda^2},2],$$
and finally by Remark \ref{remark}, we prove Theorem \ref{orange}.
\begin{lemma}\label{gap6}
Suppose that  $(\lambda, c)$ is in the orange region of Figure 3. 
Then $$\sqrt{K}+\sqrt{K}\supset [\sqrt{\lambda-\lambda^2}+\sqrt{c-\lambda^2},\sqrt{\lambda}+\sqrt{c}]\cup \lbrack \sqrt{\lambda -\lambda ^{2}}+\sqrt{1-\lambda },\sqrt{\lambda }+1] \cup \lbrack \sqrt{c-\lambda
^{2}}+\sqrt{1-\lambda },2].$$
\end{lemma}
\begin{proof}
By Lemma \ref{y1} and Lemma \ref{sqrt1}, we  let 
$$I=J=[\lambda -\lambda ^{2},\lambda ]\cup \lbrack c-\lambda ^{2},c]\cup
\lbrack 1-\lambda ,1-\lambda +\lambda c]\cup \lbrack 1-\lambda ^{2},1]. 
$$
Therefore, 
$$f(I,J)=\cup_{i=1}^{10} H_i, H_i=[i_l, i_r], 1\leq i\leq 10,$$
where 
\begin{eqnarray*}
H_1&=&[2\sqrt{\lambda -\lambda ^{2}},2\sqrt{
\lambda }]=[1_l, 1_r]\\
H_2&=&\lbrack \sqrt{\lambda -\lambda ^{2}}+\sqrt{c-\lambda ^{2}},
\sqrt{\lambda }+\sqrt{c}]=[2_l, 2_r] \\
 H_3&=& \lbrack \sqrt{\lambda -\lambda ^{2}}+\sqrt{1-\lambda },\sqrt{\lambda }
+\sqrt{1-\lambda +\lambda c}]=[3_l, 3_r]\\
 H_4&=& \lbrack \sqrt{\lambda -\lambda ^{2}}+\sqrt{
1-\lambda ^{2}},\sqrt{\lambda }+1]=[4_l, 4_r] \\
 H_5&=& \lbrack 2\sqrt{c-\lambda ^{2}},2\sqrt{c}]=[5_l, 5_r]\\
  H_6&=& \lbrack \sqrt{c-\lambda
^{2}}+\sqrt{1-\lambda },\sqrt{c}+\sqrt{1-\lambda +\lambda c}]=[6_l, 6_r]\\
 H_7&=& \lbrack 
\sqrt{c-\lambda ^{2}}+\sqrt{1-\lambda ^{2}},\sqrt{c}+1]=[7_l, 7_r] \\
 H_8&=& \lbrack 2\sqrt{1-\lambda },2\sqrt{1-\lambda +\lambda c}]=[8_l, 8_r]\\
  H_9&=& \lbrack 
\sqrt{1-\lambda }+\sqrt{1-\lambda ^{2}},\sqrt{1-\lambda +\lambda c}+1]=[9_l, 9_r]\\
  H_{10}&=&
\lbrack 2\sqrt{1-\lambda ^{2}},2]=[10_{l}, 10_{r}]
\end{eqnarray*}
Note that $1-\lambda\geq c \geq c^2\geq (1-c-\lambda)^2$ (see Lemma \ref{useful}). Then by  Lemmas \ref{sqrt1} and \ref{lem}
and let $I=[1-\lambda, 1]$,  we have $$\sqrt{K}+\sqrt{K}\subset  \lbrack 2\sqrt{1-\lambda },2].$$
Therefore,  $$H_8\cup  H_9\cup  H_{10}=\lbrack 2\sqrt{1-\lambda },2].$$
By the necessary condition of $\sqrt{K}+\sqrt{K}=[0,2]$, i.e. 
$$1+\sqrt{c}-2\sqrt{1-\lambda }\geq 0,$$
it follows that 
$$H_7\cup H_8\cup  H_9\cup  H_{10}=\left\lbrack \sqrt{c-\lambda ^{2}}+\sqrt{1-\lambda ^{2}},2\right].$$
Now, we prove 
$$6_{r}-7_{l}=\sqrt{c}+\sqrt{1-\lambda +\lambda c}-(\sqrt{c-\lambda ^{2}}+
\sqrt{1-\lambda ^{2}})\geq 0.$$
However, in the orange region of Figure 3, we have 
$$\sqrt{y}+\sqrt{1-x +x y}-(\sqrt{y-x^{2}}+
\sqrt{1-x^{2}})\geq 0.$$
Hence,$$H_6 \cup  H_7\cup H_8\cup  H_9\cup  H_{10}=\lbrack \sqrt{c-\lambda
^{2}}+\sqrt{1-\lambda },2].$$
Similarly, we are able to prove
$$3_{r}-4_{l}=\sqrt{\lambda }+\sqrt{1-\lambda +\lambda c}-(\sqrt{\lambda
-\lambda ^{2}}+\sqrt{1-\lambda ^{2}})\geq 0.$$
Therefore, 
$$H_3\cup H_4=[\sqrt{\lambda -\lambda ^{2}}+\sqrt{1-\lambda }, \sqrt{\lambda}+1].$$
\end{proof}
\begin{remark}\label{remark}
There is a gap between $H_2$ and $H_3$,  that is, generally in the  orange region of Figure 3, $2_r<3_l.$
Similarly, there is a gap between $H_4$ and $H_6$, i.e. we may have $ 4_r<6_l.$
We need  to cover these two gaps by other intervals containing in $\sqrt{K}+\sqrt{K}$, therefore, if these two intervals are covered, then we have that 
$$[0,2]\supset \sqrt{K}+\sqrt{K}\supset \cup_{k=0}^{\infty}(\sqrt{\lambda})^k[\sqrt{\lambda -\lambda ^{2}}+\sqrt{c-\lambda^2 },2]\cup \{0\}=[0,2],$$
where we use the inequality 
$$2\sqrt{\lambda}\geq \sqrt{\lambda -\lambda ^{2}}+\sqrt{c-\lambda^2 }.$$
However, in the orange region (Figure 4), we have 
$$2\sqrt{x}\geq \sqrt{x -x^{2}}+\sqrt{y-x^2 }.$$
Therefore, it remains to cover two gaps. 
\end{remark}
Now we  need to cover  the following gaps, i.e. $$[\sqrt{\lambda}+\sqrt{c}, \sqrt{\lambda-\lambda^2}+\sqrt{1-\lambda}], [\sqrt{\lambda}+1, \sqrt{c-\lambda^2}+\sqrt{1-\lambda}]$$   in terms of the thrid-level basic intervals.

To prove Theorem \ref{orange},  we shall find some interval in $\sqrt{K}+\sqrt{K}$ which covers the interval $$[\sqrt{\lambda}+\sqrt{c}, \sqrt{\lambda-\lambda^2}+\sqrt{1-\lambda}].$$

\begin{figure}
  \centering
  \includegraphics[width=350pt]{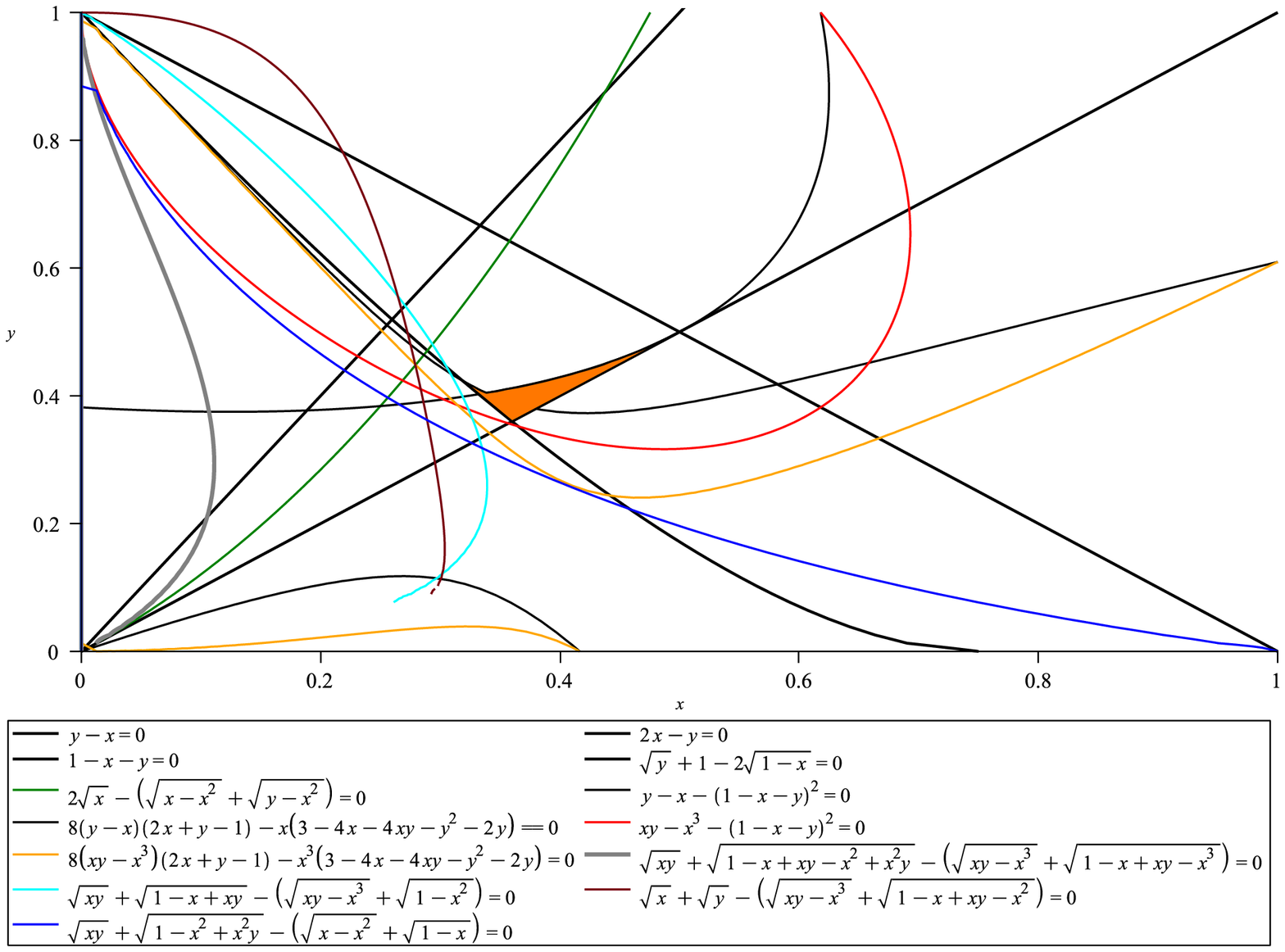}
   \caption{ }
\end{figure}
\begin{figure}
  \centering
  \includegraphics[width=350pt]{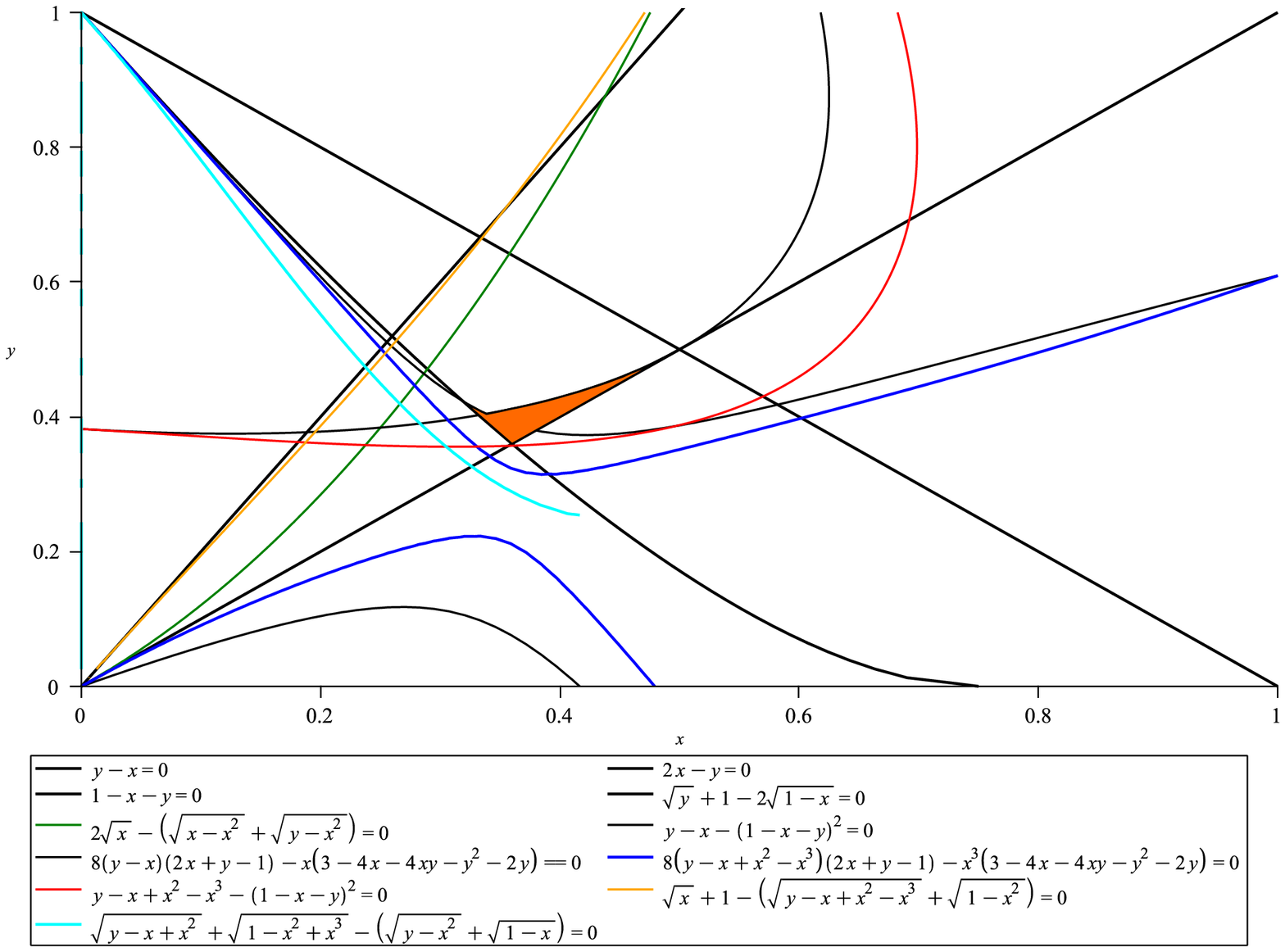}
   \caption{ }
\end{figure}
\begin{lemma}\label{gap1}
In the   orange region of Figure 4, we have 
\begin{equation*}
\left\lbrace\begin{array}{cc}
                c\lambda -\lambda ^{3}\geq (1-c-\lambda
)^{2},\\
             8(c\lambda -\lambda ^{3})(2\lambda +c-1)\geq \lambda ^{3}(3-4\lambda
-4\lambda c-c^{2}-2c)
                \end{array}\right.
\end{equation*}
\end{lemma}
\begin{proof}
In Figure 4, we note that the remaining region, i.e. the orange region,   satisfies the condition 
\begin{equation*}
\left\lbrace\begin{array}{cc}
                yx-x^{3}\geq (1-y-x
)^{2},\\
             8(xy -x^{3})(2x+y-1)\geq x^{3}(3-4x
-4xy-y^{2}-2y).
                \end{array}\right.
\end{equation*}
Therefore, we are able to use Lemma \ref{sqrt1}. 
\end{proof}
\begin{lemma}\label{gap1}
In the   orange region of Figure 3,  i.e. $(\lambda,c)$ satisfies the following conditions
\begin{equation*}
\left\lbrace\begin{array}{cc}
                c\lambda -\lambda ^{3}\geq (1-c-\lambda
)^{2},\\
             8(c\lambda -\lambda ^{3})(2\lambda +c-1)\geq \lambda ^{3}(3-4\lambda
-4\lambda c-c^{2}-2c)
                \end{array}\right.
\end{equation*}
Then $\sqrt{K}+\sqrt{K}\supset [\sqrt{\lambda c-\lambda ^{3}}+
\sqrt{1-\lambda +\lambda c-\lambda ^{2}},\sqrt{\lambda c}+\sqrt{1-\lambda
^{2}+\lambda ^{2}c}]$.
\end{lemma}
\begin{proof}
 Let $$ I=[\lambda c-\lambda ^{3},\lambda c]$$
$$J=[1-\lambda +\lambda c-\lambda ^{2},1-\lambda +\lambda c-\lambda
^{2}+\lambda ^{2}c]\cup \lbrack 1-\lambda +\lambda c-\lambda ^{3},1-\lambda
+\lambda c]\cup \lbrack 1-\lambda ^{2},1-\lambda ^{2}+\lambda ^{2}c].$$
Therefore,  by Lemmas \ref{gap1} and \ref{sqrt1}, it follows that
\begin{eqnarray*}
f(I,J)&=&[\sqrt{\lambda c-\lambda ^{3}}+\sqrt{1-\lambda +\lambda c-\lambda
^{2}},\sqrt{\lambda c}+\sqrt{1-\lambda +\lambda c-\lambda^2+\lambda^2c}
]\\
&\cup &[\sqrt{\lambda c-\lambda ^{3}}+\sqrt{1-\lambda +\lambda c-\lambda
^{3}},\sqrt{\lambda c}+\sqrt{1-\lambda +\lambda c}
]\\
&\cup & \lbrack \sqrt{\lambda c-\lambda ^{3}}+\sqrt{1-\lambda ^{2}},\sqrt{
\lambda c}+\sqrt{1-\lambda ^{2}+\lambda ^{2}c}].
\end{eqnarray*}
From Figure 4, we have 
$$\sqrt{\lambda c}+\sqrt{1-\lambda +\lambda c-\lambda
^{2}+\lambda ^{2}c}\geq \sqrt{\lambda c-\lambda ^{3}}+\sqrt{1-\lambda
+\lambda c-\lambda ^{3}}$$
 and 
$$\sqrt{\lambda c}+\sqrt{1-\lambda +\lambda c}\geq \sqrt{
\lambda c-\lambda ^{3}}+\sqrt{1-\lambda ^{2}}.$$
Therefore, by Lemmas \ref{gap1}, \ref{sqrt1} and \ref{lem}, it follows that 
$$\sqrt{K}+\sqrt{K}\supset f(I,J)=[\sqrt{\lambda c-\lambda ^{3}}+
\sqrt{1-\lambda +\lambda c-\lambda ^{2}},\sqrt{\lambda c}+\sqrt{1-\lambda
^{2}+\lambda ^{2}c}].$$
\end{proof}
\begin{lemma}\label{gap5}
In the orange region of Figure 4, 
$$[\sqrt{\lambda c-\lambda ^{3}}+
\sqrt{1-\lambda +\lambda c-\lambda ^{2}},\sqrt{\lambda c}+\sqrt{1-\lambda
^{2}+\lambda ^{2}c}]\supset [\sqrt{c}+\sqrt{\lambda}, \sqrt{\lambda-\lambda^2}+\sqrt{1-\lambda}]$$
\end{lemma}
\begin{proof}
In Figure 4, the if $(\lambda, c)$ is in the orange region, then 
$$\sqrt{\lambda c-\lambda ^{3}}+\sqrt{1-\lambda +\lambda
c-\lambda ^{2}}\leq \sqrt{\lambda }+\sqrt{c}$$
and 
$$\sqrt{\lambda c}+\sqrt{1-\lambda ^{2}+\lambda ^{2}c}\geq \sqrt{\lambda
-\lambda ^{2}}+\sqrt{1-\lambda }.$$
\end{proof}
\begin{lemma}\label{gap2}
In the   orange region of Figure 5, we have 
\begin{equation*}
\left\lbrace\begin{array}{cc}
                 c-\lambda +\lambda ^{2}-\lambda ^{3}\geq (1-c-\lambda )^{2},\\
           8(c-\lambda +\lambda ^{2}-\lambda ^{3})(2\lambda +c-1)\geq \lambda
^{3}(3-4\lambda -4\lambda c-c^{2}-2c).
                \end{array}\right.
\end{equation*}
\begin{proof}
In Figure 5, we note that  the orange region   satisfies the condition 
\begin{equation*}
\left\lbrace\begin{array}{cc}
                 y-x +x^{2}-x^{3}\geq (1-y-x )^{2},\\
           8(y-x +x^{2}-x^{3})(2x +y-1)\geq x
^{3}(3-4x -4x y-y^{2}-2y).
                \end{array}\right.
\end{equation*}
Therefore, we are able to use Lemma \ref{sqrt1}. 
\end{proof}
\end{lemma}
Now, we consider the gap between $H_4= \lbrack \sqrt{\lambda -\lambda ^{2}}+\sqrt{
1-\lambda ^{2}},\sqrt{\lambda }+1]$ and $H_6=\lbrack \sqrt{c-\lambda
^{2}}+\sqrt{1-\lambda },\sqrt{c}+\sqrt{1-\lambda +\lambda c}],$ i.e. we shall cover the interval $[\sqrt{\lambda}+1, \sqrt{c-\lambda^2}+\sqrt{1-\lambda}]$ by some subset of $\sqrt{K}+\sqrt{K}.$
\begin{lemma}\label{gap3}
In the  orange region of Figure 5, 
$$[\sqrt{c-\lambda +\lambda ^{2}-\lambda ^{3}}+\sqrt{1-\lambda ^{2}},\sqrt{
c-\lambda +\lambda ^{2}}+\sqrt{1-\lambda ^{2}+\lambda ^{3}}]\supset [\sqrt{\lambda}+1, \sqrt{c-\lambda^2}+\sqrt{1-\lambda}]. $$
\end{lemma}
\begin{proof}
By Lemmas \ref{gap2},  \ref{sqrt1}  and \ref{lem}   we  cover the  $[\sqrt{\lambda}+1, \sqrt{c-\lambda^2}+\sqrt{1-\lambda}]$ 
by
$$f(I,J)=[\sqrt{c-\lambda +\lambda ^{2}-\lambda ^{3}}+\sqrt{1-\lambda ^{2}},\sqrt{
c-\lambda +\lambda ^{2}}+\sqrt{1-\lambda ^{2}+\lambda ^{3}}],$$
where $$I=[c-\lambda +\lambda ^{2}-\lambda ^{3},c-\lambda +\lambda ^{2}] \mbox{ and }
J=[1-\lambda ^{2},1-\lambda ^{2}+\lambda ^{3}]$$
it remains  to prove
\begin{equation*}
\left\lbrace\begin{array}{cc}
                \sqrt{c-\lambda +\lambda ^{2}-\lambda ^{3}}+\sqrt{1-\lambda ^{2}}\leq \sqrt{%
\lambda }+1,\\
          \sqrt{c-\lambda ^{2}}+\sqrt{1-\lambda }\leq \sqrt{c-\lambda
+\lambda ^{2}}+\sqrt{1-\lambda ^{2}+\lambda ^{3}}
                \end{array}\right.
\end{equation*}
However, from the Figure 5 we know the inequalities hold. 
\end{proof}
\begin{proof}[\textbf{Proof of Theorem \ref{orange}}]
The proof of Theorem \ref{orange} follows from 
Lemma \ref{gap6},  Remark \ref{remark}, and Lemmas \ref{gap1}, \ref{gap5}, and  \ref{gap3}. 
\end{proof}
\begin{proof}[\textbf{Proof of Theorem \ref{Main1}}]
 Theorem \ref{Main1} follows from 
Theorems \ref{orange} and \ref{blue}.
\end{proof}
\section{Final remarks}
\noindent We pose the following question
\begin{question}
Give a necessary and sufficient condition such that $\dfrac{K}{K}=[0,\infty).$
\end{question}
  \section*{Acknowledgements}
  The author would like to thank the anonymous referees for many suggestions and remarks.
The work is supported by National Natural Science Foundation of China 
(Nos. 11701302, 11671147). The work is also supported by K.C. Wong Magna Fund in Ningbo University.


\end{document}